\documentclass[11pt,epsfig]{article}

\usepackage{amsmath, amssymb, amsthm}
\usepackage[margin=1.1in]{geometry}

\usepackage[pdftex]{graphicx}

\newtheorem{thm}{Theorem}[section]
\newtheorem{lem}{Lemma}[section]
\newtheorem{rem}{Remark}[section]
\newtheorem*{rem*}{Remark}
\newtheorem{prop}{Proposition}[section]
\newtheorem{defn}{Definition}[section]
\newtheorem{cor}{Corollary}[section]

\title 
{Energy-dissipation for time-fractional phase-field equations}
\date{}

\author
{Dong Li
\thanks
{Department of Mathematics, the Hong Kong University of Science \& Technology, Clear Water Bay, Kowloon, Hong Kong.
 Email: {mpdongli@gmail.com}
 }\qquad
{Chaoyu Quan}	
\thanks{SUSTech International Center for Mathematics, Southern University of Science and Technology, Shenzhen, China.
Email: {quancy@sustech.edu.cn}
}\qquad
{Jiao Xu}	
\thanks{Department of Mathematics, Southern University of Science and Technology,	Shenzhen, China.
Email: {xuj7@sustech.edu.cn} 
}}

\begin{document}
\maketitle
\begin{abstract}
We consider a class of time-fractional phase field models including the Allen-Cahn and Cahn-Hilliard
equations.  We establish several weighted positivity results 
for functionals driven by the Caputo time-fractional derivative.  Several novel criterions are examined for
showing the positive-definiteness of the associated kernel functions. We deduce strict energy-dissipation for a number of non-local energy functionals, thereby proving fractional energy
dissipation laws.
\end{abstract}

\section{Introduction}
In this work we consider the following time-fractional phase-field models:
\begin{align} \label{001}
\partial_t^{\alpha} \phi = - \frac {\delta \mathcal E }{\delta \phi},
\end{align}
where $0<\alpha<1$, $\frac {\delta \mathcal E}{ \delta \phi}$ denotes the variational derivative
of some energy functional $\mathcal E$ in a suitable Hilbert space to be specified later, and $\partial_t^{\alpha}$ is the Caputo fractional derivative
defined by
\begin{align}
(\partial_t^{\alpha} \phi)(t) 
:= \frac 1 {\Gamma(1-\alpha)} \int_0^t  \frac {\partial_s \phi} {(t-s)^{\alpha} } ds.
\end{align}
Here $\Gamma(\cdot)$ is the usual Gamma function.  The pre-factor is introduced so that
for smooth $\phi$,  $\partial_t^{\alpha}$ will coincide with the classic derivative $\partial_t$
when $\alpha \to 1$.  We shall consider a class of
 energy functionals associated with the standard Allen-Cahn and Cahn-Hilliard models. 
More precisely, the functional $\mathcal E $ takes the form
\begin{align} \label{005} 
\mathcal E  = \int_{\Omega} ( \frac 12 \nu |\nabla \phi|^2 + F(\phi) ) dx,
\end{align}
where $\nu>0$ corresponds to the constant mobility coefficient, and $F(\phi)$ is a prototypical double-well 
potential function, commonly chosen as $F(\phi) = (1-\phi^2)^2/4$.   Sometimes to emphasize
the competition between the gradient term and the potential term, one works with the
rescaled energy
\begin{align}
\mathcal {\tilde E} = \int_{\Omega} ( \frac 12 \epsilon^2 |\nabla \phi|^2 +
\frac 1 {\epsilon^2} F(\phi) ) dx,
\end{align}
where $\epsilon>0$ corresponds to the  interfacial width of a typical landscape.  This
somewhat different form can be linked to \eqref{005} via suitable rescaling of time and
space. In this work we shall not touch these subtle issues and work only with
\eqref{005}.  For simplicity we consider the periodic boundary conditions, and take 
$\Omega=[-\pi, \pi]^d$ in physical dimensions $d\le 3$. Choosing $\mathcal E$
as in \eqref{005}, we have
\begin{align}
&\text{Allen-Cahn}:\quad \frac {\delta \mathcal E}{\delta \phi}
\Bigr|_{L^2} = \mu,  \quad \mu = - \nu \Delta \phi + F^{\prime}(\phi); \\
&\text{Cahn-Hilliard}: \quad \frac {\delta \mathcal E}{\delta \phi}
\Bigr|_{H^{-1}} = -\Delta \mu =-\Delta ( - \nu \Delta \phi + F^{\prime}(\phi ) ).
\end{align}
We then recast \eqref{001} into the following general form:
\begin{align} \label{009}
\begin{cases}
\partial_t^{\alpha} \phi = \mathcal G  \mu =\mathcal G ( - \nu \Delta \phi + F^{\prime} (\phi) ),
\qquad (t,x) \in (0,T]\times \Omega;\\
\phi \Bigr|_{t=0} = \phi_0,  \quad \text{in $\Omega$},
\end{cases}
\end{align}
where $\phi_0$ is the initial data, $F^{\prime}(\phi)=\phi^3 -\phi$ and
\begin{align}
&\text{Time-fractional Allen-Cahn}: \quad \mathcal G = -1; \\
&\text{Time-fractional Cahn-Hilliard}: \quad \mathcal G = -\Delta.
\end{align}
A myriad of other models such as molecular beam epitaxy models, thin film epitaxy with or without slope selections, can also be studied in our framework but we shall not include them
here for simplicity of presentation.  The system \eqref{009} will be the main object of study.
We refer to \cite{caputo2015, quan2020def, PM2013, MK2000, GroundW2017, Za2002, DCL2005} and the references therein for related literature
on physical motivations and modelling aspects of time-fractional equations. We refer to
Dong and Kim's remarkable series of papers \cite{Dong2020Lp, Dong2021app,
Dong2021time, Dong2021weig} for the regularity theory of time-fractional equations. 
 The goal of this work is to establish energy dissipation for a class of suitably-defined nonlocal energy 
 functionals.

In order to give a useful criterion for checking semi-positive definiteness of the kernel function
needed later, we introduce the following definition.

\begin{defn}[Admissible kernels] \label{def1.1}
Let $D=\{(x_1, x_2):\, 0<x_1 \ne x_2<1\}$ be the open unit square with the diagonal removed. Suppose $K:\, D \to \mathbb [0,\infty)$ is symmetric, i.e.
$K(x,y)=K(y,x)$ for any $(x,y) \in D$.  We shall say $K$ is strongly-admissible if there exists $\psi:\, (0,1) \to \mathbb R$ such that
the following hold for $\tilde K(x,y) = {K(x,y)} { \psi(x) \psi(y) }$:
\begin{itemize}
\item  $\tilde K$ can be extended as a $C^2$-function on 
$\overline {D_-}= \{(x,y): 0\le y\le x \le 1 \}$. 

\item $\partial_x \tilde K\le 0$, $\partial_y \tilde K \ge 0$ and 
$\partial_{xy} \tilde K \le 0$ for any $(x,y) \in D_-=\{(x,y):\, 0<y<x<1 \}$. 
\end{itemize}

We shall say $K$ is admissible if there exists a sequence  of strongly-admissible $K_n$ such that
\begin{align}
K(x_0, y_0) = \lim_{n\to \infty} K_n (x_0, y_0), \qquad \forall\, (x_0, y_0) \in D. \label{011}
\end{align}
\end{defn}
\begin{rem}
A simple example of strongly-admissible function is $K(x,y)= \frac y {x+\epsilon}$ with
$\epsilon>0$ for $(x,y) \in D_-$.  Taking the limit $\epsilon \to 0$ gives us the admissible function
$K(x,y)=\frac y x$ which is in $C^2(D)$ but not in  $C^2(\overline{D})$.  Another example of
admissible function
is $K(x,y) =(x-y)^{-\alpha}$, $0<\alpha<1$ for $(x,y) \in D$.  Note that to accommodate possible
singularities on the diagonal $x=y$ we choose to define admissibility on the domain $D$ instead
of the whole unit square. Yet another example of admissible of $K$ is given by
$K(x,y)= a(x) b(y)$ where $a^{\prime}\le 0$ and $b^{\prime}\ge 0$. 
\end{rem}
\begin{rem} \label{rem1.2}
A very useful fact for checking strong admissibility is as follows. Suppose $K=K(x,y):\,
D_- \to (0,\infty)$ satisfies $\partial_x K\le 0$, $\partial_y K \ge 0$ and 
$\partial_{xy}  K \le 0$ for any $(x,y) \in D_-$.  
Suppose $h:\, (0,\infty) \to (0,\infty)$ is a $C^2$ function such that $h^{\prime}\ge 0$,
$h^{\prime\prime}\ge 0$, then $K_1(x,y) := h( K(x,y) )$ satisfies
$\partial_x  K_1\le 0$, $\partial_y K_1 \ge 0$, and $\partial_{xy} K_1 \le 0$. An example
is $K_1(x,y) = K(x,y)^{\alpha_1}$ with $\alpha_1\ge 1$.
\end{rem}
\begin{rem} \label{rem1.3}
Yet another construction  similar to the preceding remark is as follows. Suppose $K=K(x,y):\,
D_- \to \mathbb R$ satisfies $\partial_x K\le 0$, $\partial_y K \ge 0$ and 
$\partial_{xy}  K \le 0$ for any $(x,y) \in D_-$.  
Suppose $h:\, \mathbb R \to [0,\infty)$ is a $C^2$ function such that $h^{\prime}\ge 0$,
$h^{\prime\prime}\ge 0$, then $K_1(x,y) := h( K(x,y) )$ satisfies
$\partial_x  K_1\le 0$, $\partial_y K_1 \ge 0$, and $\partial_{xy} K_1 \le 0$. An example
is $K_1(x,y) = \exp( C_1\cdot K(x,y) )$, where $C_1>0$.
\end{rem}

Our first result is an explicit and rather easy to check criterion for establishing semi-positive
definiteness.  For simplicity of presentation we state it on the unit square. By scaling it can be
extended to any square $(0,T)^2$, $T>0$ or the entire first positive quadrant.

\begin{thm}[Semi-positive definiteness] \label{thm1}
Let $D=\{(x_1, x_2):\, 0<x_1 \ne x_2<1\}$.  Suppose a symmetric function $K: \, D \to \mathbb [0,\infty)$ is admissible in the
sense of Definition \ref{def1.1}. Then $K$ is semi-positive definite, i.e. for any 
sequence of points $t_1$, $\cdots$, $t_m \in D$,  any $c_1$, $\cdots$, $c_m\in \mathbb R$,
we have
\begin{align}
\sum_{i=1}^m \sum_{j=1}^m K(t_i, t_j) c_i c_j \ge 0.
\end{align}
In yet other words, the matrix $(K(t_i, t_j))$ is semi-positive definite.

Moreover if $K$ is strongly-admissible with $\tilde K=K$, then 
\begin{align}
\int_0^1 \int_0^x K(x,y) \phi(y) \phi(x) dy dx \ge 0, \qquad\forall\, \phi \in C_c^{\infty}(
(0,1) ).
\end{align}
\end{thm}

The following innocuous corollary is the key to establishing energy-dissipation later.
\begin{cor} \label{cor1}
Let $0<\alpha<1$. Define
\begin{align}
\beta_1(t)=1, \quad \beta_2(t)= t^{\alpha},
\quad \beta_3(t) = (1-t)^{-\alpha}, \quad 0<t<1.
\end{align}
For any measurable $f:\, [0,1]\to \mathbb R$ satisfying
\begin{align}
\sup_{0<s\le 1} |f(s)| s^{1-\alpha} <\infty,
\end{align}
  we have 
\begin{align} \label{1.14a}
\int_0^1 \int_0^t \frac {f(s) f(t)} {(t-s)^{\alpha}} \beta_i(t) ds dt \ge 0,
\quad i=1,2,3.
\end{align}
\end{cor}

Our next result establishes the nonlocal energy dissipation for time-fractional Allen-Chan
and Cahn-Hilliard equations in physical dimensions $d\le 3$. 
\begin{thm} \label{thm2}
Let $0<\alpha<1$ and consider \eqref{009} for both the Allen-Cahn and the Cahn-Hilliard
models posed on the periodic torus $\Omega=[-\pi, \pi]^d$ in physical dimensions $d\le 3$.
Then the following hold.

\begin{enumerate}
\item \underline{Global wellposedness and regularity}: Allen-Cahn case. 
Let $\phi_0 \in H^{n_0} (\Omega)$, $n_0\ge 1$. 
Corresponding to the initial data $\phi_0$ there exists a unique time-global solution $\phi \in C_t^0 H^{n_0}$  in the sense of Definition \ref{def3.1}.  Moreover, for any $\tilde T>0$ and integer
$k\ge 0$ we have (below we employ the same notation as in
Definition \ref{def3.1}, in particular $\tilde \phi( \tilde t) =\tilde \phi(t^{\alpha}) = \phi(t)$)
\begin{align} 
\sup_{0<\tilde t \le \tilde T}
\| {\tilde t}^{\frac k 2} \tilde \phi (\tilde t, \cdot) \|_{H^{n_0+k} (\Omega)}
\le C^{(1)}_{\phi_0, \nu, \alpha, \tilde T, k, n_0} <\infty,
\end{align}
where $C^{(1)}_{\phi_0, \nu, \alpha, \tilde T, k, n_0}>0$ depends on ($\phi_0$, $\nu$, $\alpha$, $\tilde T$, $k$, $n_0$). If $\phi_0 \in H^{m_0}(\Omega)$, $m_0\ge 2$, then 
we also have $\partial_{\tilde t} \tilde \phi \in C_{\tilde t}^0 H^{m_0-2}$, and  for any
$\tilde T>0$, 
\begin{align} 
\sup_{0<\tilde t \le \tilde T}
\| \partial_{\tilde t} \tilde \phi (\tilde t, \cdot) \|_{H^{m_0-2} (\Omega)}
\le C^{(2)}_{\phi_0, \nu, \alpha, \tilde  T,  m_0} <\infty,
\end{align}
where $ C^{(2)}_{\phi_0, \nu, \alpha, \tilde T,  m_0}>0$ depends on ($\phi_0$, $\nu$, $\alpha$, $\tilde T$, $m_0$).

\item \underline{Global wellposedness and regularity}: Cahn-Hilliard case.
Let $\phi_0 \in H^{n_0} (\Omega)$, $n_0\ge 1$ and assume $\int_{\Omega} \phi_0 dx =0$. 
Corresponding to the initial data $\phi_0$ there exists a unique time-global solution $\phi \in C_t^0 H^{n_0}$  in the sense of Definition \ref{def3.2}.  Moreover, for any $\tilde T>0$ and integer
$k\ge 0$ we have 
\begin{align} 
\sup_{0<\tilde t \le \tilde T}
\| {\tilde t}^{\frac k 4} \tilde \phi (\tilde t, \cdot) \|_{H^{n_0+k} (\Omega)}
\le C^{(3)}_{\phi_0, \nu, \alpha, \tilde T, k, n_0} <\infty,
\end{align}
where $C^{(3)}_{\phi_0, \nu, \alpha, \tilde T, k, n_0}>0$ depends on ($\phi_0$, $\nu$, $\alpha$, $\tilde T$, $k$, $n_0$). If $\phi_0 \in H^{m_0}(\Omega)$, $m_0\ge 4$, then 
we also have $\partial_{\tilde t} \tilde \phi \in C_{\tilde t}^0 H^{m_0-4}$, and  for any
$\tilde T>0$, 
\begin{align} 
\sup_{0<\tilde t \le \tilde T}
\| \partial_{\tilde t} \tilde \phi (\tilde t, \cdot) \|_{H^{m_0-4} (\Omega)}
\le  C^{(4)}_{\phi_0, \nu, \alpha, \tilde  T,  m_0} <\infty,
\end{align}
where $C^{(4)}_{\phi_0, \nu, \alpha, \tilde T,  m_0}>0$ depends on ($\phi_0$, $\nu$, $\alpha$, $\tilde T$, $m_0$).

\item \underline{Boundedness of the usual energy}. Let $\phi_0\in H^1( \Omega)$.
For both the Allen-Cahn and the
Cahn-Hilliard (for Cahn-Hilliard we assume $\int_{\Omega} \phi_0 dx =0$) case, it holds that
\begin{align} \label{1.15A}
\mathcal E (\phi(t) ) \le \mathcal E(\phi_0), \qquad \forall\, t>0,
\end{align} 
where $\mathcal E$ is defined in \eqref{005}.

\item \underline{Nonlocal dissipation of the usual  energy}. Assume $\phi_0
\in H^{m_0}(\Omega)$, $m_0\ge 2$ for Allen-Cahn, and $\phi_0 \in H^{m_0}(\Omega)$, $m_0\ge 4$ with $\int_{\Omega} \phi_0 dx =0$ for Cahn-Hilliard.
For both Allen-Cahn and Cahn-Hilliard, it holds that
\begin{align} \label{1.21A}
\partial_t^{\alpha} \Bigl( \mathcal E( \phi(t) )  \Bigr) <0, \qquad \forall\, t>0.
\end{align}

\item \underline{Monotonic dissipation of the nonlocal energy}. Assume
$\phi_0 \in H^{n_0} (\Omega)$, $n_0\ge 1$. For Cahn-Hilliard we assume also
$\int_{\Omega} \phi_0 dx =0$. 
Suppose $w: \, (0,1) \to [0, \infty)$
is in $L^1((0,1))$. Define
\begin{align} \label{1.22A}
K(x,y) =\frac {w(x) x} {(x-y)^{\alpha}},  \quad (x,y) \in D_-;
\end{align}
and $K(y,x) = K(x,y)$ for $(x,y) \in D_-$.  Assume $K$ is admissible in the sense of 
Definition \ref{def1.1}. Define 
\begin{align}
E_{\omega}(t) = \int_0^1 \omega(\theta) \mathcal E ( \phi(\theta t) ) d\theta.
\end{align}
Then
\begin{align} \label{1.24A}
 E_{\omega}(t_2) \le E_{\omega}(t_1),  \qquad\forall\, 0\le t_1 \le t_2 <\infty.
\end{align}
\end{enumerate}
\end{thm}
\begin{rem}
To check admissibility of $K$ in \eqref{1.22A}, some sufficient conditions on $\omega$ can be given. 
For example if $\omega(\theta) \theta^{1-\alpha} (1-\theta)^{\alpha}$ is non-increasing 
in $\theta$, then $K$ in \eqref{1.22A} is admissible. To see this, it suffices to write
\begin{align}
K(x,y) =  (1-x)^{-\alpha} (1-y)^{-\alpha}  \cdot  \omega(x) x^{1-\alpha} (1-x)^{\alpha}
 \cdot \frac {x^{\alpha} (1-y)^{\alpha} } { (x-y)^{\alpha} }, \quad 0<y<x<1.
\end{align}
It suffices for us to check that $K_1(x,y) = \frac {x^{\alpha} (1-y)^{\alpha} } { (x-y)^{\alpha}  }$
satisfies $\partial_x K_1 \le 0$, $\partial_y K_1\ge 0$ and $\partial_{xy} K_1 \le 0$ in $D_-$.
In view of Remark \ref{rem1.3},  we only need to consider
\begin{align}
K_2(x,y) = \frac 1 {\alpha} \log K_1 (x,y) =
\log x +\log (1-y) - \log (x-y).
\end{align}
It is easy to check that
$\partial_x K_2 \le 0$, $\partial_y K_2\ge 0$ and $\partial_{xy} K_2\le 0$ in $D_-$.
Thus the desired conclusion follows.
\end{rem}

The rest of this paper is organized as follows. In Section 2 we give the proof of Theorem
\ref{thm1} and Corollary \ref{cor1}. In Section 3 we give the proof of Theorem \ref{thm2}. 
In Section 4 we collect several novel proofs of semi postive-definiteness of various kernel functions.

\section{Proof of Theorem \ref{thm1}  and Corollary \ref{cor1} }
\begin{proof}[Proof of Theorem \ref{thm1}]
In view of the limit \eqref{011}, it suffices for us to prove Theorem \ref{thm1} under the assumption
that $K$ is strongly-admissible. Moreover we may assume $\tilde K= K$.

It suffices to show $\int K(x,y)\phi(x)\phi(y) dx dy \ge 0$ for any $\phi \in C_c^{\infty}(0,1)$. To this 
end, define $v(x) =\int_0^x \phi(\tilde x )d\tilde x$ and note that $v^{\prime} = \phi$. Then
after successive integration by parts, we obtain 
\begin{align*}
&\int_{0<y<x<1} K(x,y) v^{\prime}(x) v^{\prime}(y) dx dy =
\frac 12 K(1,0) v(1)^2 +
\frac 12 \int_0^1 v(x)^2 ( -\partial_x K(x, 0) ) dx  \notag \\
&+ \frac 12 \int_0^1 \int_0^x (-\partial_{xy} K) \cdot (v(x)-v(y))^2 dy dx  
+\frac 12 \int_0^1 (\partial_y K)(1,y) (v(1)-v(y))^2 dy. 
\end{align*}
This is clearly nonnegative.
\end{proof}

\begin{proof}[Proof of Corollary \ref{cor1}]
We discuss several cases.

Case 1: $\beta_1(t) =1$.  Let $\epsilon>0$ and define
\begin{align}
K_{\epsilon}(t,s)=  \frac 1 {(|t-s|+\epsilon)^{\alpha} }.
\end{align}
Note that on $\overline{D_-}=\{(t,s):\; 0\le s \le t \le 1 \}$, we have
$K_{\epsilon}(t,s) = \frac 1 {(t-s+\epsilon)^{\alpha} }$ which is clearly $C^2$. It
is not difficult to check that  $K_{\epsilon}$ is strongly-admissible in the sense
of Definition \ref{def1.1} (with $\psi \equiv 1$). By Theorem \ref{thm1}, we have
\begin{align}
\int_0^1 \int_0^t K_{\epsilon}(t,s) \phi(s) \phi(t) ds dt \ge 0, \qquad \forall\, \phi
\in C_c^{\infty}((0,1)).
\end{align}
The inequality \eqref{1.14a} for $i=1$ then follows from Lebesgue Dominated Convergence.

Case 2: $\beta_2(t) =t^{\alpha}$. In this case we work with
\begin{align}
K_{\epsilon}(t,s) = \frac 1 {(t-s+\epsilon)^{\alpha} } t^{\alpha}, \quad s<t.
\end{align}
The argument is similar to Case 1 and we omit the details.

Case 3: $\beta_3(t) = (1-t)^{-\alpha}$. In this case observe that
\begin{align}
(t-s)^{-\alpha} (1-t)^{-\alpha}
=(1-t)^{-\alpha} (1-s)^{-\alpha}
\Bigl( \frac {1-s} {t-s} \Bigr)^{\alpha}, \quad 0<s<t<1.
\end{align}
Clearly the function $K(t,s) = \log ({1-s }) + \log (t-s)$ satisfies $\partial_t K\le 0$,
$\partial_s K\ge 0$, $\partial_{ts} K\ge 0$ on $D_-=\{(t,s): 0<s<t<1 \}$. By Remark
\ref{rem1.3} the function $K_{\alpha}(t,s) = \exp( \alpha K(t,s) )$ also satisfies these conditions.
It is then straightforward to check that we have admissibility and \eqref{1.14a}
for $i=3$ follows.
\end{proof}
\section{Proof of Theorem \ref{thm2}}
In this section we carry out the proof of Theorem \ref{thm2}. We first make the notion of 
solution more precise. Let $0<\alpha<1$ and consider
\begin{align}
\begin{cases}
\partial_t^{\alpha} w = \beta w + \sigma(t), \\
w \Bigr|_{t=0} =w_0,
\end{cases}
\end{align}
where $\beta$ is a constant, and $\sigma: \;[0,T]\to \mathbb R$ is assumed to a given bounded
measurable function. 
By analogy with the usual mild solution, we have
\begin{align}
w(t) = E_{\alpha,1}(\beta t^{\alpha} ) w_0 + \int_0^t s^{\alpha-1} E_{\alpha,\alpha}
(\beta s^{\alpha} ) \sigma (t-s) ds,
\end{align}
where  for $\alpha>0$, $\beta>0$ the Mittag-Leffler function $E_{\alpha,\beta}$ is 
given by
\begin{align}
E_{\alpha,\beta}(z) = \sum_{m=0}^{\infty} \frac {z^m} {\Gamma(\alpha m +\beta)}.
\end{align}
Now make a change of variable:
\begin{align}
&\tilde t = t^{\alpha}, \quad \tilde s = s^{\alpha} ;\\
& \tilde w (\tilde t) = \tilde w (t^{\alpha} ) =w(t), \quad \tilde \sigma (\tilde s)
= \tilde \sigma (s^{\alpha}) = \sigma (s).
\end{align}
We then obtain
\begin{align}
\tilde w (\tilde t)
& = E_{\alpha,1}(\beta \tilde t) w_0
+ \alpha^{-1} \int_0^{\tilde t}
E_{\alpha,\alpha}(\beta \tilde s) \tilde \sigma(  ( {\tilde t}^{\frac 1 {\alpha}}
-{\tilde s }^{\frac 1{\alpha} } )^{\alpha} ) d\tilde s \notag \\
& = E_{\alpha,1} (\beta \tilde t) w_0
+ \alpha^{-1} \tilde t \int_0^1 E_{\alpha,\alpha}(\beta \tilde t r)
\tilde \sigma ( \tilde t (1-r^{\frac 1{\alpha} } )^{\alpha} ) dr.
\end{align}
This new formalism facilitates the book-keeping of temporal regularity along
classical lines.

We now consider the time-fractional Allen-Cahn equation posed on the
periodic torus $\Omega= [-\pi, \pi]^d$ in physical dimensions $d\le 3$:
\begin{align}  \label{AC3.3}
\begin{cases}
\partial_t^{\alpha} \phi =  \nu \Delta \phi +\phi -\phi^3,
\qquad (t,x) \in (0,T]\times \Omega;\\
\phi \Bigr|_{t=0} = \phi_0,  \quad \text{in $\Omega$},
\end{cases}
\end{align}
where $0<\alpha<1$, $\nu>0$, and $\phi_0 \in H^{n_0} (\Omega)$, $n_0\ge 1$.

\begin{defn} \label{def3.1}
Let $T>0$. We say $\phi \in C([0,T], H^{n_0} (\Omega))$ is a solution to 
\eqref{AC3.3} if 
\begin{align} \label{e3.31}
\phi(t) = E_{\alpha,1} (t^{\alpha} (\nu \Delta +1) ) 
\phi_0 + \int_0^t s^{\alpha-1} E_{\alpha,\alpha} ( s^{\alpha}
(\nu \Delta +1) ) 
\Bigl( \phi(t-s) \Bigr)^3 ds, \qquad\forall\, 0<t \le T.
\end{align}
Equivalently for
$\tilde {\phi}(\tilde t, x) = \tilde {\phi}(t^{\alpha},x) = \phi(t,x)$, the requirement is
\begin{align} 
\tilde{\phi} (\tilde t)
= E_{\alpha,1} ( \tilde t (\nu \Delta +1) ) \phi_0
+ \alpha^{-1} \tilde t \int_0^1 E_{\alpha,\alpha}( \tilde t r (\nu \Delta +1) )
 \Bigl(\tilde {\phi}  ( \tilde t (1-r^{\frac 1{\alpha} } )^{\alpha} )  \Bigr)^3dr,
 \qquad \forall\, 0<\tilde t \le T^{\alpha}.
\end{align}
\end{defn}
\begin{rem}
The identity \eqref{e3.31} holds in the Banach space $L^2$-valued sense, i.e.
$ \phi$ can be viewed as a continuous map $[0, T^{\alpha}] \to L^2$. 
By bootstrapping one also show that the identity also holds in the $H^{n_0}$-valued
sense.
\end{rem}

\begin{thm}[Wellposedness and regularity for time-fractional AC]  \label{thm3.1}
Consider \eqref{AC3.3} with $\phi_0 \in H^{n_0} (\Omega)$, $n_0\ge 1$. 
There exists a unique time-global solution $\phi \in C_t^0 H^{n_0}$ 
corresponding to the initial data $\phi_0$ and
\begin{align}
\mathcal E (\phi(t) ) \le \mathcal E( \phi_0), \qquad\forall\, t>0.
\end{align}
  Moreover, for any $\tilde T>0$ and integer
$k\ge 0$ we have
\begin{align} \label{3.32}
\sup_{0<\tilde t \le \tilde T}
\| {\tilde t}^{\frac k 2} \tilde \phi (\tilde t, \cdot) \|_{H^{n_0+k} (\Omega)}
\le C_{\phi_0, \nu, \alpha, \tilde T, k, n_0} <\infty,
\end{align}
where $C_{\phi_0, \nu, \alpha, \tilde T, k, n_0}>0$ depends on ($\phi_0$, $\nu$, $\alpha$, $\tilde T$, $k$, $n_0$). If $\phi_0 \in H^{m_0}(\Omega)$, $m_0\ge 2$, then 
we also have $\partial_{\tilde t} \tilde \phi \in C_{\tilde t}^0 H^{m_0-2}$, and  for any
$\tilde T>0$, 
\begin{align} \label{3.33}
\sup_{0<\tilde t \le \tilde T}
\| \partial_{\tilde t} \tilde \phi (\tilde t, \cdot) \|_{H^{m_0-2} (\Omega)}
\le \tilde C_{\phi_0, \nu, \alpha, \tilde  T,  m_0} <\infty,
\end{align}
where $\tilde C_{\phi_0, \nu, \alpha, \tilde T,  m_0}>0$ depends on ($\phi_0$, $\nu$, $\alpha$, $\tilde T$, $m_0$).
\end{thm}
\begin{rem}
The regularity assumptions on the initial data can be lowered further by a slightly more involved
analysis. However to simplify the presentation we shall not dwell on these technical issues here in this work.
\end{rem}
\begin{proof}
We shall sketch the proof.

Step 1: Existence of a local solution.  We first show that for some sufficiently small
$T_0>0$, there exists a unique solution $\tilde \phi \in 
C_{\tilde t}^0 ([0,T_0], H^{n_0} (\Omega) )$.  This is achieved by defining
$\tilde \phi^{(0)} \equiv 0$, and for $n\ge 0$, 
\begin{align} 
&\tilde{\phi} ^{(n+1)} (\tilde t) 
= E_{\alpha,1} ( \tilde t (\nu \Delta +1) ) \phi_0 \notag \\
&\qquad+  \alpha^{-1} \tilde t \int_0^1 E_{\alpha,\alpha}( \tilde t r (\nu \Delta +1) )
 \Bigl( {\tilde {\phi} }^{(n)} ( \tilde t (1-r^{\frac 1{\alpha} } )^{\alpha} )  \Bigr)^3dr,
 \quad \, 0<\tilde t \le T_0. 
\end{align}
One can then show uniform boundedness in $C_{\tilde t}^0 ([0,T_0], H^{n_0} (\Omega) )$
and contraction in $C_{\tilde t}^0 L^2$. 

Step 2: Higher spatial regularity of the local solution. Here we show for any integer $k\ge 0$,
\begin{align} 
\sup_{0<\tilde t \le T_0}
\| {\tilde t}^{\frac k 2}  \nabla^k \tilde \phi (\tilde t, \cdot) \|_{H^{n_0} (\Omega)}
\lesssim 1.
\end{align}
Observe that 
\begin{align}
\sup_{r>0, \tilde t >0} \| (\tilde t r)^{\frac 12} \nabla  E_{\alpha, \alpha}
( \tilde t r (\nu \Delta +1) )  g \|_{L^2(\Omega) } \lesssim \| g \|_{L^2(\Omega)}.
\end{align}
On the RHS of \eqref{e3.31},  if $r\to 0$, one has  $\tilde t (1-r^{\frac 1 {\alpha} } )^{\alpha}
\sim \tilde t$, and for $r\to 1$ one can appeal to the maximal smoothing of $
E_{\alpha,\alpha} (\tilde t r (\nu \Delta +1) )$.  The desired result then follows from bootstrapping estimates using the above observations.

Step 3: Temporal regularity of the local solution. Here we show 
\begin{align} 
\sup_{0<\tilde t \le T_0}
\| \partial_{\tilde t} \tilde \phi (\tilde t, \cdot) \|_{H^{m_0-2} (\Omega)}
\lesssim 1.
\end{align}
By \eqref{e3.31}, we have
\begin{align} 
(\partial_{\tilde t} \tilde{\phi} )(\tilde t)
&= E_{\alpha,1} ( \tilde t (\nu \Delta +1) )  (\nu \Delta +1) \phi_0 \label{3.38} \\
& \quad+ \alpha^{-1}  \int_0^1 E_{\alpha,\alpha}( \tilde t r (\nu \Delta +1) )
 \Bigl(\tilde {\phi}  ( \tilde t (1-r^{\frac 1{\alpha} } )^{\alpha} )  \Bigr)^3dr \label{3.39} \\
&\quad+ \alpha^{-1} \int_0^1  \tilde t r (\nu \Delta +1)  E_{\alpha,\alpha}( \tilde t r (\nu \Delta +1) )
 \Bigl(\tilde {\phi}  ( \tilde t (1-r^{\frac 1{\alpha} } )^{\alpha} )  \Bigr)^3dr \label{3.40}\\
 & \qquad+\alpha^{-1} \tilde t \int_0^1 E_{\alpha,\alpha}( \tilde t r (\nu \Delta +1) )
 \Bigl( 
 3 \bigl( 
 \tilde {\phi}  ( \tilde t (1-r^{\frac 1{\alpha} } )^{\alpha} )  \bigr)^2
 (\partial_{\tilde t } \tilde \phi)(\tilde t(1-r^{\frac 1{\alpha}})^{\alpha} )
 \cdot (1-r^{\frac 1 {\alpha} })^{\alpha}
 \Bigr)dr.  \label{3.41}
\end{align}
Clearly this yields for some $0<T_1\le T_0$ sufficiently small, 
\begin{align} 
\sup_{0<\tilde t \le T_1}
\| \partial_{\tilde t} \tilde \phi (\tilde t, \cdot) \|_{H^{m_0-2} (\Omega)}
\lesssim 1.
\end{align}
To obtain the uniform estimate of $\partial_{\tilde t} \tilde \phi$ for $T_1\le \tilde t\le T_0$,
we modify \eqref{3.41} as follows. Observe that
\begin{align}
 \partial_{\tilde t} \Bigl(  \tilde \phi ( \tilde t (1-r^{\frac 1 {\alpha} } )^{\alpha} ) 
\Bigr) &= (\partial_{\tilde t} \tilde \phi) (\tilde t (1-r^{\frac 1{\alpha} })^{\alpha} ) 
\cdot (1-r^{\frac 1 {\alpha} } )^{\alpha}; \\
  \partial_{r} \Bigl(  \tilde \phi ( \tilde t (1-r^{\frac 1 {\alpha} } )^{\alpha} ) 
\Bigr) &= (\partial_{\tilde t} \tilde \phi) (\tilde t (1-r^{\frac 1{\alpha} })^{\alpha} ) 
\cdot  (1-r^{\frac 1 {\alpha} } )^{\alpha-1} \cdot (-r^{\frac 1{\alpha} -1} ) \notag \\
&= \partial_{\tilde t} \Bigl(  \tilde \phi ( \tilde t (1-r^{\frac 1 {\alpha} } )^{\alpha} ) 
\Bigr) \cdot  \Bigl( - \frac { r^{\frac 1 {\alpha}-1}} {1-r^{\frac  1 {\alpha} } } \Bigr).
\end{align}
We then rewrite the integral in \eqref{3.41} as 
\begin{align}
&\int_0^1 E_{\alpha,\alpha}( \tilde t r (\nu \Delta +1) )
 \Bigl( 
 3 \bigl( 
 \tilde {\phi}  ( \tilde t (1-r^{\frac 1{\alpha} } )^{\alpha} )  \bigr)^2
 (\partial_{\tilde t } \tilde \phi)(\tilde t(1-r^{\frac 1{\alpha}})^{\alpha} )
 \cdot (1-r^{\frac 1 {\alpha} })^{\alpha}
 \Bigr)dr \notag\\
=& \int_0^{\delta_0} E_{\alpha,\alpha}( \tilde t r (\nu \Delta +1) )
 \Bigl( 
 3 \bigl( 
 \tilde {\phi}  ( \tilde t (1-r^{\frac 1{\alpha} } )^{\alpha} )  \bigr)^2
 (\partial_{\tilde t } \tilde \phi)(\tilde t(1-r^{\frac 1{\alpha}})^{\alpha} )
 \cdot (1-r^{\frac 1 {\alpha} })^{\alpha}
 \Bigr)dr \label{3.45} \\
 & \quad+ \int_{1-\delta_0}^1 E_{\alpha,\alpha}( \tilde t r (\nu \Delta +1) )
 \Bigl( 
 3 \bigl( 
 \tilde {\phi}  ( \tilde t (1-r^{\frac 1{\alpha} } )^{\alpha} )  \bigr)^2
 (\partial_{\tilde t } \tilde \phi)(\tilde t(1-r^{\frac 1{\alpha}})^{\alpha} )
 \cdot (1-r^{\frac 1 {\alpha} })^{\alpha}
 \Bigr)dr \label{3.46} \\
 & \quad + \int_{\delta_0}^{1-\delta_0} E_{\alpha,\alpha}( \tilde t r (\nu \Delta +1) )
 \left(
\partial_r \Bigl( \bigl( 
 \tilde {\phi}  ( \tilde t (1-r^{\frac 1{\alpha} } )^{\alpha} )  \bigr)^3 \Bigr)
 \cdot ( - \frac {1-r^{\frac 1 {\alpha} } } {r^{\frac 1 {\alpha}-1} } )
 \right)dr. \label{3.47}
\end{align}
Clearly \eqref{3.45} and \eqref{3.46} can be treated by taking $\delta_0$ sufficiently small.
The term \eqref{3.47} can be estimated by a further integration by part in $r$. 
Thus we obtain uniform estimate of $\partial_{\tilde t} \tilde \phi$ for $T_1\le \tilde t\le T_0$.

Step 4: Continuation of the local solution: naive extension. Here we show that if $\phi
\in C([0,T_1], H^{n_0})$ solves \eqref{e3.31} on $0<t\le T_1$, then one can
extend the solution to the time interval $[0, T_1+\delta_1]$ for some $\delta_1>0$. 
We first rewrite \eqref{e3.31} as
\begin{align}
\phi(t) = E_{\alpha,1} (t^{\alpha} (\nu \Delta +1) ) \phi_0
+ \int_0^t (t-s)^{\alpha-1} E_{\alpha,\alpha}
( (t-s)^{\alpha} (\nu \Delta +1) ) ( \phi(s) ^3) ds.
\end{align}
For $t\ge T_1$, we make the decomposition
\begin{align}
\phi(t) & = E_{\alpha,1} (t^{\alpha} (\nu \Delta +1) ) \phi_0
+ \int_0^{T_1} (t-s)^{\alpha-1} E_{\alpha,\alpha}
( (t-s)^{\alpha} (\nu \Delta +1) ) ( \phi(s) ^3 )ds  \label{3.55}\\
& \quad +  \int_{T_1}^{t} (t-s)^{\alpha-1} E_{\alpha,\alpha}
( (t-s)^{\alpha} (\nu \Delta +1) ) ( \phi(s) ^3) ds. \label{3.56}
\end{align}
Since the pieces in \eqref{3.55} are already known, one can perform a contraction to obtain
the unknown in \eqref{3.56} by taking $\delta_1>0$ sufficiently small.

Step 5: Continuation of the local solution: non-blowup criterion. Here we show that if $\phi
\in C([0,T_2), H^{n_0})$ solves \eqref{e3.31} on $0<t<T_2$, and 
\begin{align} \label{3.57}
\sup_{0<t<T_2} \| \phi(t,\cdot) \|_{H^1} <\infty.
\end{align}
then one can
extend the solution  to the time interval $[0, T_2+\delta_2]$ for some $\delta_2>0$. 
Define 
\begin{align}
y_{T_2}
= E_{\alpha,1} (T_2^{\alpha} (\nu\Delta +1) )
\phi_0 +
\int_0^{T_2} 
(T_2-s)^{\alpha-1}
E_{\alpha,\alpha}((T_2-s)^{\alpha} (\nu \Delta+1) )
( \phi(s)^3) ds.
\end{align}
By \eqref{3.57} and Lebesgue Dominated Convergence, we have
\begin{align}
\lim_{t\to T_2} \phi(t) = y_{T_2}, \quad \text{in $H^1$}.
\end{align}
By bootstrapping estimates we then obtain $\phi \in C([0,T_2], H^{n_0})$. By using Step 4
we can then extend the solution past $T_2$.

Step 6: Global wellposedness for smooth initial data. Here we show that if $\phi_0
\in C^{\infty}$, then the corresponding solution $\phi$ exists globally in time and
\begin{align} \label{3.60}
\mathcal E ( \phi (t) ) \le \mathcal  E ( \phi_0), \qquad \forall\, t>0.
\end{align}
Since $\phi_0 \in C^{\infty}$, it is not difficult to check that
$\| \partial_{\tilde t} \tilde \phi (\tilde t, \cdot ) \|_{\infty} \lesssim 1$
for any $t>0$, and
\begin{align}
\sup_{0<t\le T} \Bigl( \| t^{1-\alpha} (\partial_t \phi)(t,\cdot) \|_{\infty} 
+ \| (\partial_t^{\alpha} \phi)(t,\cdot)\|_{\infty}  \Bigr) \lesssim 1,
\end{align}
provided $\phi$ is the solution on $[0,T]$. Now observe that
\begin{align}
\mathcal E ( \phi(0) )
-\mathcal E ( \phi (T) )
& = - \int_0^T \frac d {dt} \Bigl( \mathcal E ( \phi(t) ) \Bigr) dt \notag \\
& = \int_0^T \int_{\Omega} \partial_t^{\alpha} \phi \partial_t \phi dx dt \notag \\
& = \frac 1{\Gamma(1-\alpha)} \int_{\Omega}  \int_0^T \int_0^t
\frac {(\partial_t \phi)(s,x) (\partial_t \phi)(t,x) } { (t-s)^{\alpha}}
ds dt dx\ge 0,
\end{align}
where the last inequality follows from Corollary \ref{cor1}.
Thus \eqref{3.60} follows and we can extend the solution globally in time. 

Step 7: Global wellposedness for $H^{n_0}$, $n_0\ge 1$ initial data.  
Here we show that if $\phi_0
\in H^{n_0}$, then the corresponding solution $\phi$ exists globally in time and
\begin{align}  \label{3.64}
\mathcal E ( \phi (t) ) \le \mathcal  E ( \phi_0), \qquad \forall\, t>0.
\end{align}
To show this, we take $\phi_0^{(n)} \in C^{\infty}(\Omega)$ such that
$\|\phi_0^{(n)} \|_{H^1} \le 2 \| \phi_0 \|_{H^1}$, 
$\phi_0^{(n)} \to \phi_0$ in $H^1(\Omega)$ as $n\to \infty$.
Denote by $\phi^{(n)}$ the solution corresponding to the initial data 
$\phi_0^{(n)}$.    Clearly
\begin{align}
\mathcal E ( \phi^{(n)} (t) ) \le \mathcal  E ( \phi_0^{(n)} )
\lesssim \mathcal E ( \phi_0), \qquad \forall\, t>0.
\end{align}
By a simple $H^1$ stability estimate, it is not difficult to check that
$\phi^{(n)} \to \phi $ in $C_t^0 H^1$.   The estimate \eqref{3.64} then
easily follows. Further bootstrapping gives the needed spatial and temporal
regularity results. We omit further details.
\end{proof}

Next we consider the time-fractional Cahn-Hilliard equation posed on the
periodic torus $\Omega= [-\pi, \pi]^d$ in physical dimensions $d\le 3$:
\begin{align}  \label{CH3.48}
\begin{cases}
\partial_t^{\alpha} \phi =  -\nu \Delta^2 \phi -\Delta \phi +\Delta(\phi^3),
\qquad (t,x) \in (0,T]\times \Omega;\\
\phi \Bigr|_{t=0} = \phi_0,  \quad \text{in $\Omega$},
\end{cases}
\end{align}
where $0<\alpha<1$, $\nu>0$, and $\phi_0 \in H^{n_0} (\Omega)$, $n_0\ge 1$ with
$\int_{\Omega} \phi_0 dx =0$.

\begin{defn} \label{def3.2}
Let $T>0$. We say $\phi \in C([0,T], H^{n_0} (\Omega))$ is a solution to 
\eqref{CH3.48} if 
\begin{align} 
\phi(t) = E_{\alpha,1} (t^{\alpha} (-\nu \Delta^2 -\Delta) ) 
\phi_0 + \int_0^t s^{\alpha-1} \Delta E_{\alpha,\alpha} ( s^{\alpha}
(-\nu \Delta^2 -\Delta) ) 
\Bigl( \phi(t-s)^3\Bigr) ds, \qquad\forall\, 0<t \le T.
\end{align}
Equivalently for
$\tilde {\phi}(\tilde t, x) = \tilde {\phi}(t^{\alpha},x) = \phi(t,x)$, the requirement is

\begin{align} 
\tilde{\phi} (\tilde t)
= E_{\alpha,1} ( \tilde t (-\nu \Delta^2 -\Delta) ) \phi_0
+ \alpha^{-1} \tilde t \int_0^1 \Delta E_{\alpha,\alpha}( \tilde t r (-\nu \Delta^2 -\Delta) )
 \Bigl(\tilde {\phi}  ( \tilde t (1-r^{\frac 1{\alpha} } )^{\alpha} )^3  \Bigr)dr,
 \qquad \forall\, 0<\tilde t \le T^{\alpha}.
\end{align}
\end{defn}

\begin{thm}[Wellposedness and regularity for time-fractional CH] \label{thm3.2}
Consider \eqref{CH3.48} with $\phi_0 \in H^{n_0} (\Omega)$, $n_0\ge 1$ and $\int_{\Omega} \phi_0 dx =0$. 
There exists a unique time-global solution $\phi \in C_t^0 H^{n_0}$ 
corresponding to the initial data $\phi_0$ and 
\begin{align} \label{3.69}
\mathcal E (\phi(t) ) \le \mathcal E( \phi_0), \qquad\forall\, t>0.
\end{align}
 Moreover, for any $\tilde T>0$ and integer
$k\ge 0$ we have
\begin{align} 
\sup_{0<\tilde t \le \tilde T}
\| {\tilde t}^{\frac k 4} \tilde \phi (\tilde t, \cdot) \|_{H^{n_0+k} (\Omega)}
\le C_{\phi_0, \nu, \alpha, \tilde T, k, n_0} <\infty,
\end{align}
where $C_{\phi_0, \nu, \alpha, \tilde T, k, n_0}>0$ depends on ($\phi_0$, $\nu$, $\alpha$, $\tilde T$, $k$, $n_0$). If $\phi_0 \in H^{m_0}(\Omega)$, $m_0\ge 4$, then 
we also have $\partial_{\tilde t} \tilde \phi \in C_{\tilde t}^0 H^{m_0-4}$, and  for any
$\tilde T>0$, 
\begin{align} 
\sup_{0<\tilde t \le \tilde T}
\| \partial_{\tilde t} \tilde \phi (\tilde t, \cdot) \|_{H^{m_0-4} (\Omega)}
\le \tilde C_{\phi_0, \nu, \alpha, \tilde  T,  m_0} <\infty,
\end{align}
where $\tilde C_{\phi_0, \nu, \alpha, \tilde T,  m_0}>0$ depends on ($\phi_0$, $\nu$, $\alpha$, $\tilde T$, $m_0$).
\end{thm}
\begin{proof}
The proof is similar to Theorem \ref{thm3.1}. For the energy bound \eqref{3.69}, we note that 
$\phi(t,\cdot)$ has zero mean for $t>0$. By using mollification of initial data and $H^1$
stability we can assume $\phi_0$ is smooth and has mean zero. One can check that
$\partial_t \phi $ has mean zero and $(-\Delta)^{-1} \partial_t \phi $ is well-defined. 
The following computation can be rigorously justified:
\begin{align}
\mathcal E (\phi(0) )
-\mathcal E ( \phi(T) )
& = \int_0^T \int_{\Omega}
\partial_t^{\alpha} \phi  (-\Delta)^{-1} \partial_t \phi dx dt \notag \\
& = \int_0^T \int_{\Omega}
\partial_t^{\alpha} (|\nabla|^{-1} \phi )  \partial_t |\nabla|^{-1} \phi dx dt \ge 0.
\end{align}
 We omit further pedestrian details.
\end{proof}

We now complete the proof of Theorem \ref{thm2}.

\begin{proof}[Proof of Theorem \ref{thm2}]
The first three statements follow from Theorem \ref{thm3.1} and \ref{thm3.2}.  
For \eqref{1.21A} we focus on the Allen-Cahn case. The proof for the Cahn-Hilliard
case is similar. Now for $t_0>0$,
\begin{align}
-\Gamma(1-\alpha)
\partial_t^{\alpha} ( \mathcal E ( \phi (t) ) ) \Bigr|_{t=t_0}
& = \int_0^{t_0} \frac 1 {(t_0-t)^{\alpha} }
\int_{\Omega}
\partial_t^{\alpha} \phi \partial_t \phi dx dt \notag \\
& = \frac 1 {\Gamma(1-\alpha)}
\int_{\Omega}
\int_0^{t_0} \frac 1{(t_0-t)^{\alpha} }
\int_0^t \frac { (\partial_t \phi)(t,x) (\partial_t \phi)(s, x) }
{(t-s)^{\alpha} } ds dt dx \ge 0,
\end{align}
where the last inequality follows from a rescaled version of Corollary \ref{cor1}.
By using the regularity estimates derived earlier, it is not difficult to justify the
above computation rigorously. Thus \eqref{1.21A} follows.
The proof of \eqref{1.24A} follows along similar lines. By mollification and $H^1$ stability we may assume the initial data $\phi_0$ is smooth. It then suffices for us to check
$\frac d {dt} E_{\omega}(t)<0$ for any $t>0$. The proof is straightforward.  We omit further
details.
\end{proof}

\section{Several other proofs for positive-definiteness}
In this section we collect several proofs for semi-positive definiteness of the kernel functions
mentioned in preceding sections.  For example, for any $u\in C[0,T]$, and
$\alpha \in (0,1)$, it holds that
\begin{align*}
\int_0^t \int_0^s
\frac{s^{\alpha} u(s) u(\tau)} { (t-s)^{\alpha} (s-\tau)^{\alpha}}
d\tau ds \ge 0, \quad\,\forall\, t\in (0,T].
\end{align*}
Now we give a short argument. 
First note that by scaling we may assume $t=1$. It then suffices to consider
the kernel 
\begin{align*}
k(s,\tau) = (s-\tau)^{-\alpha} \cdot(\frac {s} {1-s} )^{\alpha}, 
\quad \text{for $1>s>\tau>0$}.
\end{align*}
One can  extend the kernel to the whole domain by regarding
$s=\max\{s, \tau\}$. 

\begin{thm}
$k(s,\tau)$ is a positive definite kernel on $[0,1]^2$.
\end{thm}

\begin{proof}
Observe the identity ($s>\tau$):
\begin{align*}
(s-\tau)^{-\alpha}
\cdot(\frac s {1-s} )^{\alpha} 
= \frac 1 {(1-s)^{\alpha} (1-\tau)^{\alpha}}
\cdot (1- \frac {1-s}{1-\tau} \cdot \frac {\tau} s)^{-\alpha} 
\end{align*}
One can then expand into a positive power series and use Lemma \ref{lem1}
below.

\end{proof}

\begin{lem} \label{lem1}
$\tau/s$ ($s>\tau$) is\footnote{Symmetrize in the usual way.} a positive definite kernel, similarly 
$\frac {1-s} {1-\tau}$ ($s>\tau$) is also
a positive definite kernel. Consequently $\frac {\tau}s \cdot  \frac {1-s} {1-\tau}$ is a positive
definite kernel.\footnote{Note that point-wise product of two postive-definite kernels
is still a positive-definite kernel!}
\end{lem}
\begin{proof}
Observe $\tau/s = e^{-|\log s- \log\tau|}$. The result then follows from
the fact that $e^{-|x-y|}$ is a positive definite kernel and a log change of variable.
\end{proof}

\subsection{A much shorter proof}
Rewrite the identity ($s>\tau$):
\begin{align*}
&(s-\tau)^{-\alpha}
\cdot(\frac s {1-s} )^{\alpha} 
= \frac 1 {(1-s)^{\alpha} (1-\tau)^{\alpha}}
\cdot (1- \frac {\frac{\tau}{1-\tau} }  {\frac s {1-s}} 
)^{-\alpha}  \notag \\
=&\; \frac 1 {(1-s)^{\alpha} (1-\tau)^{\alpha}} 
\cdot (1-e^{-|X(s)-X(\tau) |} )^{-\alpha},
\end{align*}
where $X(s)=\ln \frac  s {1-s} $. 

One can then use the following fact:
If $F$ has a positive power series expansion, then 
$F(e^{-|X(s)-X(\tau) |})$ is positive definite.
Indeed one can use the identity
\begin{align*}
e^{-|z|} = \frac 1 {2\pi} \int_{-\infty}^{\infty} \frac 1 {1+\xi^2}
e^{i \xi \cdot z} d\xi
\end{align*}
and obtain
\begin{align*}
e^{-|X(s)-X(t)|} = \frac 1 {2\pi} \int_{-\infty}^{\infty} \frac 1 {1+\xi^2}
e^{i \xi \cdot  X(s) }  e^{-i \xi \cdot X(t)} d\xi.
\end{align*}
Clearly then
\begin{align*}
\int\int 
e^{-|X(s)-X(t)|} u(s) u(t) ds dt &= \frac 1 {2\pi} \int_{-\infty}^{\infty} \frac 1 {1+\xi^2}
 \int u (s) e^{i \xi \cdot  X(s) } ds \int u(t) e^{-i \xi \cdot X(t)}dt d\xi \notag \\
 &= \frac 1 {2\pi} \int_{-\infty}^{\infty} \frac 1 {1+\xi^2}
|U(\xi)|^2 d\xi.
 \end{align*}
One should note that we did not use any regularity property of the map $X(s)$!
(But keep in mind that monotonicity was used when we made a change of variable
$(\frac {\tau} {1-\tau})/ (\frac s {1-s}) \to \exp( -|X(s)-X(t)|)$!)

\begin{rem*}
One may wonder whether $(1-\frac{ \phi(\tau) } {\phi(s) } )^{-\alpha}$ can be replaced by a more
general function.  This can indeed be done. For example consider 
$G(z)= F(e^{-|z|} )$ where $F$ is monotonically increasing on $(0,\infty)$ 
with $F^{\prime}>0$, $F^{\prime\prime}>0$ (basically think of $F(y)
\sim 1+y +y^2+\cdots$ which is in some sense ``completely positive"). Then we see that
when $0<z<\infty$,
\begin{align*}
G^{\prime} = F^{\prime}\cdot  (-e^{-z}) <0, \\
G^{\prime\prime}= F^{\prime\prime} e^{-2z} + F^{\prime} e^{-z}>0.
\end{align*}

By Polya's criterion
one can write
\begin{align*}
 \hat G(\xi) = 2 \int_0^{\infty} F(e^{-|z|})
 \cos 2\pi \xi z dz \ge 0.
\end{align*}
In yet other words $G(z)$ can be represented as 
\begin{align*}
G(z) = \int e^{2\pi i \xi \cdot z} \hat G(\xi) d\xi,
\end{align*}
where $\hat G(\xi) \ge 0$.

\end{rem*}

\begin{rem*}
One may wonder how to prove Polya's criterion directly. Here is one simple idea assuming
$f$ is nice and has sufficient decay.
What we have in mind is that $f$ looks like $e^{-|x|}$ on $(0,\infty)$. 
Consider $\xi>0$, 
\begin{align*}
\frac 12\hat f(\xi) &= \int_0^{\infty} f(x) \cos x \xi dx  \notag \\
&=f(x) \frac {\sin x\xi} {\xi} \Bigr|_{x=0}^{\infty}
- \frac 1 {\xi} \int_0^{\infty} f^{\prime}(x)  \sin \xi x dx \notag \\
&=  \int_0^{\infty} f^{\prime\prime}(x) \frac 
{1-\cos \xi x} {\xi^2} dx \ge 0.
\end{align*}
Actually one should see that the second integration by part is not necessary.
One can plot a picture of say $\sin x$ (with $\xi=1$ just for illustration) and note
that $-f^{\prime}(x)$ is increasing. From the picture, it is obvious that within each
period the integral must be positive!

\end{rem*}

\begin{rem*}
The function $(1-\cos k)/k^2$ looks very much like the Fejer kernel in usual Fourier
series. Indeed, we have
\begin{align*}
\int_{\mathbb R} (1-|t|)_{+} e^{- it k} dt = \frac 2 {k^2} (1-\cos k).
\end{align*}
Thus
\begin{align*}
\frac {1-\cos \xi r} {(\xi r)^2} = \frac 1 2
\int_{\mathbb R} (1-|t|)_{+} e^{-it r\xi} dt=
\frac 1 {2r} \int_{\mathbb R} (1- |\frac t r|)_{+} e^{-it \xi} dt.
\end{align*}
Then returning to the previous remark, we see that
\begin{align*}
\frac 12 \hat f(\xi) &= 
\int_0^{\infty} f^{\prime\prime}(r) \frac 
{1-\cos \xi r} {\xi^2 r^2}  r^2 dr\notag \\
&= \frac 12 
\int_0^{\infty} \int_{t\in \mathbb R}
f^{\prime\prime}(r) (1- |\frac t r|)_{+} e^{-i t \xi} r dr dt.
\end{align*}
In yet other words, we have
\begin{align*}
f(t) = \int_0^{\infty} f^{\prime\prime}(r)
(r-|t|)_+ dr = \int_0^{\infty} (1- \frac{|t|} r)_+  (r f^{\prime\prime}(r) ) dr.
\end{align*}
Denote $\nu$ a measure on $(0, \infty)$ as $\nu(dr) = r f^{\prime\prime}(r) dr$.
Note that $\int_0^{\infty} r f^{\prime\prime}(r) dr =1$ so that $\nu$ is a probability
measure on $(0,\infty)$.  Then we recover a version of Polya's criterion stated  in Durret's book 
(see \cite{Durret}, Page 137, Theorem 3.3.22) as follows:

\textbf{Theorem}: Let $\phi(t)$ be real nonnegative and have
$\phi(0)=1$, $\phi(t)=\phi(-t)$, and $\phi$ is decreasing and convex on $(0,\infty)$
with $\lim_{t \to 0+} \phi(t)=1$, $\lim_{t\to \infty} \phi(t)=0$. Then there
is a probability measure $\nu$ on $(0,\infty)$, so that
\begin{align*}
\phi(t) = \int_0^{\infty} ( 1 - |t/s|)_+ \nu (ds).
\end{align*}

\end{rem*}

\begin{rem*}
From the graph, we see that if $f$ is convex, then the representation
\begin{align*}
f(t) = \int_0^{\infty} (r-|t|)_+ f^{\prime\prime}(r) dr
\end{align*}
is expressing $f$ as a nonnegative combination of ``triangle functions"
$(r-|t|)_+$ which is intuitively clear! In some sense these triangle functions are the
simplest convex functions! Another way to see it is to regard the above as
some kind of distributional computation, namely
\begin{align*}
\partial_{rr} (  (r-t)_+ )=\partial_r( 1_{r>t}) = \delta(r-t).
\end{align*}

\end{rem*}

\subsection{Yet another short proof for the positivity of $\tau/s$}
\begin{prop}
$\int_{0<\tau<s<1} \frac{\tau} s u(\tau) u(s) d \tau ds \ge 0$ 
for any $u \in C_c^{\infty}((0,1))$.
\end{prop}
\begin{proof}
\begin{align*}
\int_0^1 \frac 1 {s^2}\partial_s \Bigl(  \frac 12  ( \int_0^s \tau u(\tau) d\tau)^2 \Bigr) ds
=  \int_0^1 s^{-3} ( \int_0^s \tau u(\tau) d\tau)^2 ds.
\end{align*}
\end{proof}

Similarly if $\partial_s (  {\phi(s) A(s) } ) \ge 0$, then 
\begin{align*}
\int_{0<\tau<s<1}
\frac{\phi(\tau)} {A(s)} 
u(\tau) u(s) d\tau ds = \int_0^1(-\frac 12
\partial_s(\frac 1 {\phi(s) A(s) } )    (\int_0^s \phi(\tau) u(\tau) d\tau)^2 ds \ge 0.
\end{align*}

In particular if $A=\phi$ and $\phi\phi^{\prime}\ge 0$, then we have positivity!

\begin{rem*}
In yet other words, the condition is
\begin{align*}
\frac {\phi(\tau)} {A(s)} = \frac {\phi(\tau)\phi(s)} {A(s) \phi(s)}
= \frac {\phi(\tau) \phi(s)} { \text{monotone increasing}}.
\end{align*}
\end{rem*}

\begin{rem*}
In our case we can take $\phi(s)=s/(1-s)$ for $s\in (0,1)$ which clearly holds true.
In the Brownian bridge case we have (the kernel is
$\tau \wedge s -\tau s = (1-s)\tau$ when $\tau<s$) 
$\phi(\tau)=\tau$, $A(s)= \frac 1 {1-s}$ which also
holds true.
\end{rem*}
\subsection{Connection with Brownian bridge}
In this sub-section we explain the connection of the semi-positive definiteness of the kernel functions
to the classical theory of Brownian bridge, namely the classical fact that  $(1-s)t$ is the covariance of Brownian bridge. Recall a Brownian bridge on $[0,1]$ is
\begin{align*}
X_t = B_t -t B_1,
\end{align*}
where $B_t$ is the standard Brownian motion. One can then check that
(say $t<s$)
\begin{align*}
\mathbb E X_t X_s = \mathbb E (B_t-tB_1) (B_s-sB_1)=
t\wedge s -st-ts+ts = t \wedge s - ts  = (1-s)t,
\end{align*}
where we recall $\mathbb E B_t B_s= t \wedge s$. 

In yet other words, the kernel $(1-s)t$ is precisely the covariance kernel
of the Brownian bridge and thus must be positive definite!
Indeed, trivially one sees that
\begin{align*}
\int  (t\wedge s- st) u(t) u(s) ds dt = \int \mathbb E(X(t) u(t)  X(s) u(s) ) ds dt
= \mathbb E    ( \int u(t) X(t) dt )^2.
\end{align*}

\subsection{Connection with a Dirichlet BVP}
 Consider the Dirichlet problem on $[0,1]$:
\begin{align*}
-\partial_{tt}u =- u^{\prime\prime} =\delta(t-s),
\end{align*}
where $0<s<1$ is fixed and we employ Dirichlet boundary condition. 
 This is just the standard Green's function! Clearly $u$ must be a straight line
 from $0$ to $s$ and then back to $1$. So
 \begin{align*}
 u(t)= \begin{cases}
 k t, \quad 0<t<s; \\
 (k-1)t -(k-1), \quad s<t<1.
 \end{cases}
 \end{align*}
 By continuity one then get
 \begin{align*}
 ks = (k-1) s - (k-1).
 \end{align*}
 Thus $k=1-s$.  
 
 Now alternatively one can use Fourier series to solve the same problem. Namely
 \begin{align*}
 u= \sum_{n=1}^{\infty}  u_n \sin n\pi t,
 \end{align*}
 and
 \begin{align*}
 u_n (n\pi)^2 = 2 \sin n \pi s.
 \end{align*}
Thus 
\begin{align*}
u = \sum_{n=1}^{\infty} \frac 2 {(n\pi)^2} \sin n \pi t \sin n\pi s.
\end{align*}
In yet other words
\begin{align*}
(-\partial_{tt})_{\operatorname{Dirichlet} }^{-1}  \delta(t-s) =
t\wedge s -ts = \sum_{n=1}^{\infty}
\frac 2 {(n\pi)^2}
\sin n \pi s \sin n \pi t.
\end{align*}

\begin{rem}
 One can also obtain the Fourier expansion for the kernel $K(s,t)=
t\wedge s -ts $ directly by expanding it in terms of the basis
$\sin m \pi t \sin n\pi s$. Note that $K(1,t)=K(0,t)=0=K(s,0)=K(s,1)$!
\end{rem}

A natural question is whether we can prove directly the positive definiteness
of the kernel function $s\wedge t - ts$ on $[0,1]\times [0,1]$. Note that
\begin{align*}
s\wedge t - t s = ts (   (\frac 1s ) \wedge (\frac 1t ) -1).
\end{align*}
Thus it suffices to prove for $A, B \in [1,\infty)$, the kernel function
$A\wedge B-1$ is positive definite. This amounts to showing that
(say in the case $n=4$) for every $0< t_1 \le t_2 \le t_3\le t_4$, we have
\begin{align*}
\begin{pmatrix}
t_1 \quad t_1 \quad t_1 \quad t_1 \\
t_1  \quad t_2 \quad t_2 \quad t_2 \\
t_1 \quad t_2  \quad t_3 \quad t_3\\
t_1 \quad t_2 \quad t_3\quad t_4
\end{pmatrix}
\succ t_1 ,
\end{align*}
where the notation $\succ $ means positive definite and $t_1$ denotes the
constant matrix whose entries are all $t_1$. This can be easily done by induction,
since the statement for $n=4$ obviously reduces to showing
\begin{align*}
\begin{pmatrix}
t_2-t_1 \quad t_2-t_1 \quad t_2-t_1\\
t_2-t_1\quad t_3-t_1 \quad t_3-t_1\\
t_2-t_1 \quad t_3-t_1 \quad  t_4-t_1
\end{pmatrix}
\succ 0
\end{align*}
which clearly follows from the $n=3$ statement:

\begin{align*}
\begin{pmatrix}
t_2-t_1 \quad t_2-t_1 \quad t_2-t_1\\
t_2-t_1\quad t_3-t_1 \quad t_3-t_1\\
t_2-t_1 \quad t_3-t_1 \quad  t_4-t_1
\end{pmatrix}
\succ (t_2-t_1) \succ 0.
\end{align*}

 An integration-by-parts proof is also possible. Note that
for $0<s,t<1$:
\begin{align*}
&K(s,t)= s\wedge t - st = \frac 12 (s+1) - \frac 12{|s-t|} -st; \\
& \partial_s K(s,t)= \frac 12 - \frac 12 \operatorname{sgn}(s-t) -t,
\quad (\partial_s K)(s,0)= 0=(\partial_s K)(s,1);\\
& \partial_t \partial_s K(s,t) = \delta(t-s) -1.
\end{align*}
Then writing $v(s)= \int_0^s u(r)dr$, we obtain
\begin{align*}
\int K(s,t) u(s)u (t) ds dt &=- \int \partial_s K(s,t) v(s) u(t) ds dt \notag\\
& = \int  \partial_{st} K(s,t) v(s) v(t) ds dt \notag \\
& = \| v\|_{L^2([0,1])}^2 - \| v\|_{L^1([0,1])}^2.
\end{align*}
\begin{rem*}
Instead of the special kernel here, a more general kernel condition can also be worked out if we keep track of the boundary terms and so on.
\end{rem*}
\begin{rem*}
The proof here is short but perhaps is not the heart of the matter. One can use the approach
we developed earlier to obtain more refined  quantitative estimates. We shall not dwell on this
issue here.
\end{rem*}

One should note that the distribution computation above can be made rigorous by 
observing the following fact: say $k$ is continuous
and $k(0)=k(1)=0$,  $f\in C^{\infty}([0,1])$, then
\begin{align*}
&\int_0^1 k(s) \partial_s ( f \eta(\frac s {\epsilon}) ) ds
=\operatorname{NICE}+ \int_0^1 k(s) f \frac 1 {\epsilon} \eta^{\prime}(\frac s {\epsilon})
ds \to 0,\quad\text{as $\epsilon \to 0$};\\
&\int_0^1 k(s) \partial_s ( f \eta(\frac {s-1} {\epsilon}) ) ds
=\operatorname{NICE}+ \int_0^1 k(s) f \frac 1 {\epsilon} \eta^{\prime}(\frac {s-1}{\epsilon})
dx \to 0,\quad\text{as $\epsilon \to 0$},
\end{align*}
where $\eta$ is a cut-off bump function such that $\eta(x)=1$ for $|x|\le 1$
and $\eta(x) =0$ for $|x|>2$. Here we used the fact that $k(0)=k(1)=0$.

\begin{rem*}
Yet another short proof is as follows
\begin{align*}
\int_{0<\tau<s<1}  (1-s)s^{-1} ( \int_0^s \tau u(\tau) d\tau s u(s)) ds
= \int_0^1 \frac 1 2 ( s^{-2})  ( \int_0^s \tau u(\tau) d\tau)^2 ds.
\end{align*}
\end{rem*}

\frenchspacing
\bibliographystyle{plain}

\end{document}